\RequirePackage[l2tabu,orthodox]{nag}

\documentclass[headsepline,footsepline,footinclude=false,oneside,fontsize=11pt,paper=a4,listof=totoc,bibliography=totoc,DIV=12]{scrbook} 

\PassOptionsToPackage{table,svgnames,dvipsnames}{xcolor}

\usepackage[utf8]{inputenc}
\usepackage[T1]{fontenc}
\usepackage[sc]{mathpazo}
\usepackage[english]{babel} 
\usepackage[autostyle]{csquotes}
\usepackage{graphicx}
\usepackage{scrhack} 
\usepackage{listings}
\usepackage{lstautogobble}
\usepackage{tikz}
\usetikzlibrary{shapes,arrows,chains}
\usepackage{pgfplots}
\usepackage{pgfplotstable}
\usepackage{booktabs} 
\usepackage[final]{microtype}
\usepackage{caption}
\usepackage{hyperref}
\usepackage{ifthen} 
\usepackage{pdftexcmds} 
\usepackage{paralist} 
\usepackage{subfig} 
\usepackage{siunitx} 
\usepackage{multirow} 
\usepackage{array} 
\usepackage{makecell} 
\usepackage{pdfpages} 
\usepackage{adjustbox} 
\usepackage{tablefootnote} 
\usepackage{threeparttable} 

\usepackage{amssymb}
\usepackage{pifont}

\usepackage[acronym,xindy,toc]{glossaries} 
\makeglossaries

\newacronym[shortplural={D$_{\text{dye}}$}, longplural={donor dye, ex. Alexa 488}]{ddye}{D$_{\text{dye}}$}{donor dye, ex. Alexa 488}

\newacronym[description={\glslink{r0}{F\"{o}rster distance}}]{R0}{$R_{0}$}{F\"{o}rster distance}

\newglossaryentry{r0}
{
  name=\glslink{R0}{\ensuremath{R_{0}}},
  text=F\"{o}rster distance,
  description={F\"{o}rster distance, where 50\% ...}, 
  sort=R
}

\newglossaryentry{kdeac}
{
  name=\glslink{R0}{\ensuremath{k_{DEAC}}},
  text=$k_{DEAC}$, 
  description={is the rate of deactivation from ... and emission)}, 
  sort=k
}

\newacronym[shortplural={TUM}, longplural={Technical University of Munich}]{tuma}{TUM}
{Technical University of Munich}

\newglossaryentry{tum}
{
  name = TUM,
  description = {A university in Germany, Bavaria, in the city of Munich},
  plural = TUM
}

\newglossaryentry{computer}
{
  name=computer,
  description={is a programmable machine that receives input,
               stores and manipulates data, and provides
               output in a useful format}
} 

\usepackage[numbers]{natbib}

\setkomafont{disposition}{\normalfont\bfseries} 
\BeforeTOCHead[toc]{{\cleardoublepage\pdfbookmark[0]{\contentsname}{toc}}}

\usepackage{amsthm}
\theoremstyle{remark}
\newtheorem{examplex}{Example}
\newenvironment{eg}
  {\pushQED{\qed}\examplex}
  {\popQED\endexamplex}
 
\theoremstyle{remark}
\newtheorem{remark}{Remark}

\theoremstyle{theorem}
\newtheorem{proposition}{Proposition}
\newtheorem{theorem}{Theorem}
\newtheorem{lemma}{Lemma}
\newtheorem{definition}{Definition}

\usepackage{xspace}
\renewcommand{\vert}{\ | \ }
\usepackage{mathrsfs}
\usepackage{enumitem}
\usepackage{makecell}
\renewcommand{\u}[1]{\underline{#1}}
\usepackage{algorithm}
\usepackage{algpseudocode}
\usepackage{centernot}
\usepackage{mathtools}
\usepackage{stmaryrd}
\usepackage{cleveref} 

\newcommand{\downmapsto}{\rotatebox[origin=c]{-90}{$\mapsto$}\mkern2mu}

\definecolor{TUMBlue}{HTML}{0065BD}
\definecolor{TUMSecondaryBlue}{HTML}{005293}
\definecolor{TUMSecondaryBlue2}{HTML}{003359}
\definecolor{TUMBlack}{HTML}{000000}
\definecolor{TUMWhite}{HTML}{FFFFFF}
\definecolor{TUMDarkGray}{HTML}{333333}
\definecolor{TUMGray}{HTML}{808080}
\definecolor{TUMLightGray}{HTML}{CCCCC6}
\definecolor{TUMAccentGray}{HTML}{DAD7CB}
\definecolor{TUMAccentOrange}{HTML}{E37222}
\definecolor{TUMAccentGreen}{HTML}{A2AD00}
\definecolor{TUMAccentLightBlue}{HTML}{98C6EA}
\definecolor{TUMAccentBlue}{HTML}{64A0C8}

\pgfplotsset{compat=newest}
\pgfplotsset{
  cycle list={TUMBlue\\TUMAccentOrange\\TUMAccentGreen\\TUMSecondaryBlue2\\TUMDarkGray\\},
}

\graphicspath{
    {logos/}
    {figures/}
}

\newcolumntype{P}[1]{>{\centering\arraybackslash}p{#1}} 
\newcolumntype{M}[1]{>{\centering\arraybackslash}m{#1}} 
\newcolumntype{L}[1]{>{\raggedright\arraybackslash}m{#1}} 
\newcolumntype{R}[1]{>{\raggedleft\arraybackslash}m{#1}} 

\newcommand*{\getUniversityA}{Vietnamese--German University}
\newcommand*{\getFacultyA}{Department of Computer Science}
\newcommand*{\getUniversityB}{Frankfurt University of Applied Sciences}
\newcommand*{\getFacultyB}{Department of Informatics}
\newcommand*{\getTitle}{Combinatorial Statistics on\\\vspace{5mm} Pattern-avoiding Permutations}

\newcommand*{\getAuthor}{Thien Hoang}
\newcommand*{\getDoctype}{Bachelor's Thesis in Computer Science}
\newcommand*{\getSupervisor}{Dr. Huong Tran}
\newcommand*{\getAdvisor}{Prof. Dr. Manuel Clavel}
\newcommand*{\getSubmissionDate}{24.09.2021}
\newcommand*{\getSubmissionLocation}{Binh Duong, Vietnam}

\begin{document}

\selectlanguage{english}

\pagenumbering{alph}
\begin{titlepage}
  \oddsidemargin=\evensidemargin\relax
  \textwidth=\dimexpr\paperwidth-2\evensidemargin-2in\relax
  \hsize=\textwidth\relax

  \centering

  \begin{minipage}[t]{0.4\textwidth}
    \centering
    \includegraphics[height=15mm]{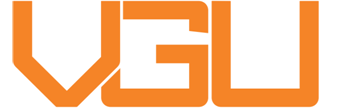}
  \end{minipage}
  \begin{minipage}[t]{0.4\textwidth}
    \centering
    \includegraphics[height=20mm]{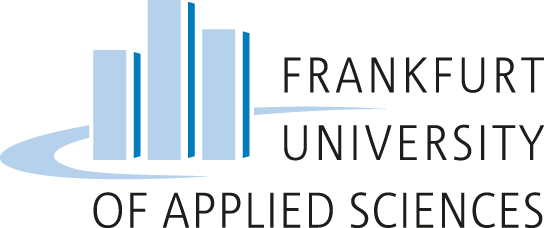}
  \end{minipage}

  \vspace{15mm}
  {\LARGE\MakeUppercase{\getUniversityA{}}}\\

  \vspace{5mm}
  {\large\MakeUppercase{\getFacultyA{}}}\\

  \vspace{7mm}
  \emph{and}
  \vspace{7mm}

  {\LARGE\MakeUppercase{\getUniversityB{}}}\\

  \vspace{5mm}
  {\large\MakeUppercase{\getFacultyB{}}}\\

  \vspace{15mm}
  {\Large \getDoctype{}}

  \vspace{15mm}
  {\huge\bfseries \getTitle{}}

  \vspace{40mm}
  {\LARGE \getAuthor{}}

\end{titlepage}

\frontmatter{}

\begin{titlepage}
  \centering

  \begin{minipage}[t]{0.4\textwidth}
    \centering
    \includegraphics[height=15mm]{logos/vgu.png}
  \end{minipage}
  \begin{minipage}[t]{0.4\textwidth}
    \centering
    \includegraphics[height=20mm]{logos/frauas.png}
  \end{minipage}

  \vspace{15mm}
  {\LARGE\MakeUppercase{\getUniversityA{}}}\\

  \vspace{5mm}
  {\large\MakeUppercase{\getFacultyA{}}}\\

  \vspace{7mm}
  \emph{and}
  \vspace{7mm}

  {\LARGE\MakeUppercase{\getUniversityB{}}}\\

  \vspace{5mm}
  {\large\MakeUppercase{\getFacultyB{}}}\\

  \vspace{15mm}
  {\Large \getDoctype{}}

  \vspace{15mm}
  {\huge\bfseries \getTitle{}}

  \vspace{15mm}
  \begin{tabular}{l l}
    Author:          & \getAuthor{} \\
    Supervisor:      & \getSupervisor{} \\
    Co-supervisor:         & \getAdvisor{} \\
    Submission Date: & \getSubmissionDate{} \\
  \end{tabular}

\end{titlepage}

\cleardoublepage{}

\thispagestyle{empty}
\vspace*{0.8\textheight}
\noindent
\makeatletter
\ifthenelse{\pdf@strcmp{\languagename}{english}=0}
{I confirm that this \MakeLowercase{\getDoctype{}} is my own work and I have documented all sources and material used.}
{Ich versichere, dass ich diese \getDoctype{} selbstständig verfasst und nur die angegebenen Quellen und Hilfsmittel verwendet habe.}
\makeatother

\vspace{15mm}
\noindent
\getSubmissionLocation{}, \getSubmissionDate{} \hspace{50mm} \getAuthor{}

\cleardoublepage{}

\thispagestyle{empty}

\vspace*{5cm}

\begin{center}
\usekomafont{section}
This thesis is dedicated to \emph{myself}.\\
\vspace{12cm}
\emph{``Studying Mathematics for its own sake is equally as valid and equally as worthwhile, even when there is no immediate practical purpose.''}\\
\vspace{5mm}
\raggedleft
------- James Grime
\end{center}

\vspace{10mm}

\cleardoublepage{}

{\addcontentsline{toc}{chapter}{Acknowledgments}}
\thispagestyle{empty}

\vspace*{20mm}

\begin{center}
\usekomafont{section} {\Huge\textsc{Acknowledgments}}

{\large\vspace{15mm}

I would like to thank my supervisor, Dr. Huong Tran, for her incredibly helpful guidance during my working on the thesis.}

\vspace{10mm}

{\large I also want to express my sincere gratitude to my family for their unconditional support; to the friends I made in PiMA and in the university, for inspiring and shaping the person who I am today.}

\end{center}

\vspace{10mm}


\cleardoublepage{}

\chapter{\abstractname}

The study of Mahonian statistics dated back to 1915 when \citet{mahon} showed that the major index and the inverse number have the same distribution on a set of permutations with length $n$. Since then, many Mahonian statistics have been discovered and much effort have been done to find the equidistribution between two Mahonian statistics on permutations avoiding length-3 classical patterns. In recent years, \citet{amini} and \citet{do} have done extensive research with various methods to prove the equidistributions, ranging from using generating functions, Dyck paths, block decompositions, to bijections. In this thesis, we will solve the conjectured equidistribution between bast and foze on Av(312) using the bijection method, as well as refine two established results in \cite{do} with a combinatorial approach.

\microtypesetup{protrusion=false}
\tableofcontents{}
\microtypesetup{protrusion=true}

\mainmatter{}

\newcommand{\inv}{\mathrm{inv}}
\newcommand{\Asc}{\mathrm{Asc}}
\newcommand{\asc}{\mathrm{asc}}
\newcommand{\Atop}{\mathrm{Atop}}
\newcommand{\Abot}{\mathrm{Abot}}
\newcommand{\Des}{\mathrm{Des}}
\newcommand{\des}{\mathrm{des}}
\newcommand{\Dtop}{\mathrm{Dtop}}
\newcommand{\Dbot}{\mathrm{Dbot}}
\newcommand{\Rmax}{\mathrm{Rmax}}
\newcommand{\rmax}{\mathrm{rmax}}
\newcommand{\Rmaxl}{\mathrm{Rmaxl}}
\newcommand{\Lmax}{\mathrm{Lmax}}
\newcommand{\lmax}{\mathrm{lmax}}
\newcommand{\Lmaxl}{\mathrm{Lmaxl}}
\newcommand{\Lmin}{\mathrm{Lmin}}
\newcommand{\Lminl}{\mathrm{Lminl}}
\newcommand{\Rmin}{\mathrm{Rmin}}
\newcommand{\rmin}{\mathrm{rmin}}
\newcommand{\Rminl}{\mathrm{Rminl}}
\newcommand{\maj}{\mathrm{maj}}
\newcommand{\makl}{\mathrm{makl}}
\newcommand{\bast}{\mathrm{bast}}
\newcommand{\foze}{\mathrm{foze}}
\newcommand{\Av}{\mathrm{Av}}
\newcommand{\calA}{\mathscr{A}}
\newcommand{\calB}{\mathscr{B}}


\chapter{Introduction}\label{chapter:introduction}

According to \citet{rowntree}, the term ``statistics'' bears at least four meanings: the discipline of statistics; the methods to collect, process, or interpret data; the data collected; and lastly, the special figures calculated from the data. In this thesis, we will be looking at \textit{combinatorial statistics}, which is closest to the fourth meaning; and the ``data'' we will be dealing with is \textit{permutations}. A distribution of a statistic is a summary of all possible values it can induce on a permutation set and the frequencies of those values. In particular, we are interested in the way some pair of statistics are distributed evenly (i.e. \textit{equidistributed}) on some pair of permutation sets.

\textit{Mahonian statistics} is a class of combinatorial statistics in which they are distributed evenly on $S_n$ (set of all permutations with length $n$). One representative Mahonian statistic is the \textit{major index}, whose Mahonity was shown by \citet{mahon}. However, when we put some restriction such as the permutations must avoid some \textit{pattern} $\sigma$, they may not be equidistributed anymore. The equidistribution problem on pattern-avoiding permutation sets $S_n(\sigma)$ is more worthwhile when $\sigma$ is a classical pattern of length 3, because $|S_n(\sigma)|$ is the $n$-th Catalan number. That means, the results we have on $S_n(\sigma)$ can be extended to other Catalan objects.

In each equidistribution problem, we are provided with two statistics and two pattern-avoiding sets. Over the past decades, researchers have found various Mahonian statistics (see \Cref{tab:mahon} for some examples). Additionally, there are six classical patterns of length 3. A quick multiplication can tell us that there are a huge number of equidistribution problems to be solved. \citet{amini} have done an extensive research on them and presented a lot of conjectured equidistributions. There are many ways to prove an equidistribution. \citet{amini} has used block decomposition, Dyck paths, and generating functions in their proofs. \citet{do} have shown more results by using the bijection method.

In this thesis, we will use the bijection method to prove a new theorem, while reimagine some solutions in \cite{do}. We have successfully proven the equidistribution between two Mahonian statistics bast and foze on the set of 312-avoiding permutations. Furthermore, using a stack-based algorithm, we have refined the bijection that transforms the statistic maj to makl on 231-avoiding permutation set. We also give a refined combinatorial proof for the equidistribution of $\foze''$ and inv on $S_n(312)$ and $S_n(321)$ ($n$ is any positive integer).

The thesis is organized in four chapters:

\begin{itemize}
    \item \textbf{\Cref{chapter:introduction}}, the one you are reading, is meant to give the readers a high-level understanding of what problems we are trying to solve and how we will solve them.
    \item \textbf{\Cref{chapter:background}} introduces some formal definitions and notations which will be used frequently in the thesis.
    \item \textbf{\Cref{chapter:mainres}} presents three equidistribution problems in three sections; two of which (\Cref{sec:majmakl} and \Cref{sec:foze''inv}) are established results with refined solutions; the other one (\Cref{sec:bastfoze}) presents a novel equidistribution we have just found.
    \item \textbf{\Cref{chapter:conclusion}} will summarize all the major and minor results presented in \Cref{chapter:mainres} and a few notes on how the research on this topic can go further.
\end{itemize}

\chapter{Combinatorial statistics and their distributions}\label{chapter:background}

In this chapter, we will go through some definitions and notations that will be used frequently throughout the course of this thesis. We will start with basic concepts such as permutation, reduced form, and patterns (\Cref{sec:pp}); then we introduce some combinatorial statistics and how they are related with vincular patterns (\Cref{sec:stats}); and lastly we address the statistical equidistribution problem on pattern-avoiding sets (\Cref{sec:equistats}).

\section{Permutation patterns: Classical patterns and Vincular patterns}\label{sec:pp}
A permutation $\pi$ of length (or size) $n$ is an arrangement of $n$ pairwise distinct and comparable letters, which are usually $\{1,2,\dots,n\}$ (denoted as $[n]$). We also say that $\pi$ is an $n$-permutation. The $i$-th entry of $\pi$ is denoted as $\pi_i$, and naturally, $i$ is the position or index of $\pi_i$. The first entry is indexed 1 (in oppose to common programming practices which use 0-based index). $\pi$ can also be seen as a function, mapping $i$ to $\pi_i$; and for that reason, sometimes we may denote $\pi(I)=\{\pi_i | i\in I\}$ where $I\subseteq [n]$.

\begin{eg}
Given $\pi=4235167$ and $I=\{1,3,5,7\}$ then $\pi(I)=\{4,3,1,7\}$.
\end{eg}

$\pi$ is said to be in reduced form if its letters are in $[n]$. Reducing a permutation is replacing its $i$-th smallest letter with $i$. This concept is useful when we define an occurrence of a pattern below. In this paper, all permutations are implied to be in reduced form, unless indicated otherwise. The set of $n$-permutation is denoted as $S_n$.

Let $\sigma\in S_m$ and $\pi\in S_n$ (where $m\leq n$). $\pi$ is said to contain the \textit{classical pattern} $\sigma$ if there exist $m$ indices $i_1 < i_2 < \dots < i_m$ such that the reduced form of $\pi_{i_1}\pi_{i_2}\dots\pi_{i_m}$ is $\sigma$. In that case, $(i_1, i_2, \dots, i_m)$ is said to be an occurrence of $\sigma$, or a $\sigma$-occurrence. If $\pi$ contains no $\sigma$-occurrences, it is said that $\pi$ \textit{avoids} $\sigma$ or $\pi$ is a $\sigma$-avoiding permutation. The number of $\sigma$-occurrences in $\pi$ is written as $\sigma(\pi)$.

\begin{eg}\label{eg:cpattern}
$(1,3,5)$ is a 321-occurrence in $\pi=4235167$ because the reduced form of $\pi_1\pi_3\pi_5=431$ is 321. To explain further, $\pi_1 > \pi_3 > \pi_5$ is similar to how $3 > 2 > 1$. The other 321-occurrence is $(1,2,5)$.
\end{eg}

A \textit{vincular pattern} is a specialisation of classical pattern, where we put some adjacency restrictions by underlining some consecutive letters in the permutation.

\begin{eg}
$321$ is a classical pattern, but $\u{32}1$ is a vincular pattern, where the positions matching `3' and `2' must be next to each other. In the last example, $\pi=4235167$ contains two 321-occurrences, but only $(1,2,5)$ is the occurrence of $\u{32}1$.
\end{eg}






\section{Combinatorial statistics}\label{sec:stats}

A \textit{combinatorial statistic}, or simply \textit{statistic}, is a map $\text{st}: S\rightarrow \mathbb{N}$. For example, $i$ is called a \textit{left-to-right maximum} of $\pi$ if $\pi_i$ is larger than any other entry appearing to the left of it. Statistic lmax is defined as the number of the left-to-right maxima of a permutation. Another kind of statistics is \textit{set-induced statistics}, usually denoted by capitalizing the first letter of their notations, e.g. Lmax as set of all left-to-right maxima. For a more comprehensive list of statistics and their definitions used in the thesis, please refer to \cref{tab:statistics}.

When a statistic is evaluated on a set of permutations, we are interested in its \textit{distribution}, i.e. how many times a value is obtained by the statistic. If we describe the distribution of st as a generating function:
$$\sum_{\pi\in S}q^{\text{st}(\pi)}$$
then the coefficient of $q^x$ tells us how many values of $\pi\in S$ there are,  such that st$(\pi)=x$.

\begin{eg}\label{eg:inv}
Let st map each permutation of $S_3$ to its number of 21-occurrences. We have the following table:

\begin{center}
    \begin{tabular}{c|c|c|c|c|c|c}
        $\pi$ & $123$ & $132$ & $213$ & $231$ & $312$ & $321$ \\
        \hline
        $\text{st}(\pi)$ & 0 & 1 & 1 & 2 & 2 & 3 
\end{tabular}
\end{center}

Its distribution can be written as $1 + 2q + 2q^2 + q^3$. In the second row, we see that 0 appears once (corresponding to the term $1q^0$), 1 appears twice (corresp. $2q^1$), 2 appears twice (corresp. $2q^2$), and 3 appears once (corresp. $1q^3$). 
\end{eg}

The statistic in Example \ref{eg:inv} is also known as the inversion number, i.e. number of pairs $(i,j)$ where $i<j$ and $\pi_i>\pi_j$. The inversion number of $\pi$ is denoted as $\inv(i)$. The distribution $\inv$ was given by \citet{rodrigues} as:
$$\sum_{\pi\in S_n}q^{\inv(pi)}=[n]_q!$$
where $[n]_q!=[n]_q[n-1]_q\dots [1]_q$ and $[k]_q= 1 + q + \dots + q^{k-1}$.

Later on, \citet{mahon} showed that the major index, defined as $\maj(\pi)=\sum_{i\in \Des(\pi)}i$, also has the same distribution on $S_n$. Any statistics that have such distribution are called \textit{Mahonian} statistics. Since then, many more Mahonian statistics were discovered (see \Cref{tab:mahon}). \citet{babson} have shown that some Mahonian statistics can be expressed by totalling the number of occurrences of some vincular patterns. For example, in \Cref{tab:mahon}, the major index is equal to the number of $1\u{32}$-occurrences, plus the number of $2\u{31}$-occurrences, plus so on. Following the notation $\sigma(\pi)$ in \Cref{sec:pp}, we have:
$$\inv(\pi)=1\u{32}(\pi) + 2\u{31}(\pi) + 3\u{21}(\pi) + \u{21}(\pi) = (1\u{32}+ 2\u{31}+ 3\u{21}+ \u{21})(\pi)$$


\begin{table}[t]
    \centering
    \begin{tabular}{c|c|c}
    \hline
        \textbf{Name} & \textbf{Sum of vincular patterns} & \textbf{Reference}\\\hline
        inv & \underline{23}1 + \underline{31}2 + \underline{32}1 + \underline{21} & \citet{rodrigues} \\
        maj & 1\underline{32} + 2\underline{31} + 3\underline{21} + \underline{21} & \citet{mahon} \\
        makl & 1\underline{32} + 2\underline{31} + \underline{32}1 + \underline{21} & \citet{clarke} \\
        bast & \underline{13}2 + \underline{21}3 + \underline{32}1 + \underline{21} & \citet{babson} \\
        foze & \underline{21}3 + 3\underline{21} + \underline{13}2 + \underline{21} & \citet{foata} \\
        foze$''$ & \underline{23}1 + \underline{31}2 + \underline{31}2 + \underline{21} & \citet{foata}\\\hline
    \end{tabular}
    \caption{Some Mahonian statistics and their expressions in vincular patterns. The naming was adopted following \citet{amini}.}
    \label{tab:mahon}
\end{table}

\section{Statistic equidistributions}\label{sec:equistats}

While all Mahonian statistics have the same distribution on $S_n$, they may distribute differently on \textit{pattern-avoiding permutation sets}, and it is of great importance two find which pair of statistics are distributed evenly on certain pair of pattern-avoiding sets. Let $\sigma$ be a pattern, we denote that $S_n(\sigma)$ consists of $\sigma$-avoiding $n$-permutations. We also write $\Av(\sigma)=\bigcup_{n\in\mathbb{N}^+}S_n(\sigma)$ as the set of all $\sigma$-avoiding permutations.

Given two statistics, st$_1$ and st$_2$, and two pattern-avoiding sets, $S_n(\sigma_1)$ and $S_n(\sigma_2)$, we say \emph{st$_1$ and st$_2$ are equidistributed on $S_n(\sigma_1)$ and $S_n(\sigma_2)$} if this equality is satisfied:

$$
\sum_{\pi\in S_n(\sigma_1)} q^{\text{st}_1(\pi)} = \sum_{\pi\in S_n(\sigma_2)} q^{\text{st}_2(\pi)}\quad (\forall n\in\mathbb{N}^+)
$$

In this thesis we will consider such equidistributions where st$_1$ and st$_2$ are Mahonian statistics and $\sigma_1$ and $\sigma_2$ are length-3 classical patterns. One motivation to study the distributions on such sets $S_n(\sigma_1)$ and $S_n(\sigma_2)$ is because their cardinality are the $n$-th Catalan number, and the equidistribution can be extended to other Catalan-objects.

Among the many ways of proving equidistribution of st$_1$ and st$_2$ on $S_n(\sigma_1)$ and $S_n(\sigma_2)$, this thesis will put a focus on the bijection method. This means that we will find a bijection $f: S_n(\sigma_1)\rightarrow S_n(\sigma_2)$ such that st$_1(\pi_1)=\text{st}_2(f(\pi_1))$. If such a bijection exists, the equidistribution is naturally implied.


There are three so-called \textit{trivial bijections} on permutations: the reverse, complement, and inverse bijection. Given $\pi\in S_n$:
\begin{itemize}
    \item Reverse: $r(\pi)=\pi_n\pi_{n-1}\dots \pi_1$.
    \item Complement: $c(\pi)=(n-\pi_1+1)(n-\pi_2+1)\dots (n-\pi_n+1)$.
    \item Inverse: The $\pi_i$-th entry of $i(\pi)$ is $i$. We also denote $\pi^{-1}=i(\pi)$.
\end{itemize}

\begin{eg}
$r(4235167) = 7615324,\ c(4235167) = 4653721,\ i(4235167)=5231467$
\end{eg}

It is natural to extend the bijections above to classical patterns too. Let $b$ be one of the bijection above, it is easy to that if $\pi$ contains $x$ occurrences of $\sigma$, where $\sigma$ is a classical pattern, then $b(\pi)$ contains $x$ occurrences occurrences of $b(\sigma)$. As a result, the property also holds true if $b$ is a composition of some trivial bijections.

As for vincular patterns, only reverse and complement are applied (see \citet{bivincular}). Inverse of a vincular pattern is a \textit{bivincular pattern}, which is out of scope of this thesis. When a vincular pattern is reversed, the adjacency restriction is reversed as well.

\begin{eg} $r(2\underline{31}) = \underline{13}2$. Notice that in $2\u{31}$ require the last two positions of the occurrence to be adjacent, but in $\u{13}2$ it is the first two positions. On the other hand,
$c(2\underline{31}) = 2\underline{13}$. Notice that the adjacent positions are unchanged when we apply the complement bijection.
\end{eg}

The trivial bijections on permutations and their extension to patterns is a fundamental tool for us to solve some equidistribution problems, particularly the pair $\bast$ and $\foze$ on $\Av(312)$ and $\Av(312)$ in \Cref{sec:bastfoze} of the next chapter.

\chapter{Main results}\label{chapter:mainres}
There are three equidistribution problems to be discussed in this chapter. In \Cref{sec:majmakl}, we will revisit the equidistribution of maj and makl on Av(231), proven in \citet{do} using bijection method. Next in \Cref{sec:bastfoze}, we will establish a new equidistribution between bast and foze on Av(312), using some results we had in \Cref{sec:majmakl}. The last sections will focus on refining the equidistribution between $\foze''$ and inv on Av(231) and Av(312), but with a combinatorial proof instead of an inductive one in \cite{do}.

\section{Revisit the equidistribution of $\maj$ and $\makl$ on Av(231)}\label{sec:majmakl}
As introduced, here we will revisit the pair maj and makl on Av(231), proven in \citet{do} using bijection method. In order to construct the bijection $\theta$ that transforms $\maj$ to $\makl$ on $\Av(231)$, we need an intermediate bijection $\theta'$ that transforms $\Asc$ to $\Atop$ on $\Av'(231)$ (set of 231-avoiding permutations beginning with the largest element). Then, $\theta$ will be defined in terms of $\theta'$. While working on the thesis, we found that their construction of $\theta'$ was incorrect, so we started writing a new implementation to confirm the correctness of \Cref{thm:ascatop} below.

\begin{theorem}\label{thm:ascatop}\emph{\citet{do}}.
There exists a bijection $\theta': \Av'(231)\rightarrow\Av'(231)$ transforming the statistic $\Asc$ to $\Atop$.
\end{theorem}

First, we will go through a couple of definitions and notations, some of which will be reused in \Cref{sec:bastfoze}, then we will present an algorithm and prove that it is a valid implementation of $\theta'$ mentioned above. By the end of this section we will briefly describe how $\theta: \Av(231)\rightarrow\Av(231)$ is defined in terms of $\theta': \Av'(231)\rightarrow\Av'(231)$.

\subsection{Definitions and notations}

\newcommand{\idr}{inverse descent run\xspace}
\newcommand{\idrs}{inverse descent runs\xspace}
Given an $n$-permutation $\pi$, a \textit{descent run} is a maximal set $\{i, i+1, \dots, j\}\subset [n]$ such that $\pi_i > \pi_{i+1} > \dots > \pi_j$, in the sense that adding $i-1 > 0$ or $j+1 \leq n$ to the set does not make it satisfy the decreasing condition. An \idr is a set $\{i_1, i_2, \dots, i_j\}$ such that $\{\pi_{i_1}, \pi_{i_2}, \dots, \pi_{i_j}\}$ is a descent run in $\pi^{-1}$. If $i_1<i_2<\dots<i_j$ then $\pi_{i_1}, \pi_{i_2}, \dots, \pi_{i_j}$ is an arithmetic sequence with common difference of -1.

Suppose that $\pi$ has $k$ \idrs $I_1, I_2, \dots, I_k$, ordered by the maximum element of each set, i.e. $\max(I_1) > \max(I_2) > \dots > \max(I_k)$. Within the scope of this thesis, all \idrs of any permutation are implied to be ordered in this way. It is also easy to see that all $I_1, I_2, \dots, I_k$ is a partition of $[n]$.

\begin{eg}
Given $\pi=7651324$, its \idrs are $I_1=\{1,2,3,7\},\ I_2=\{5,6\}$, and $I_3=\{4\}$.
\end{eg}

The concept \idr has several properties, particularly on 231-avoiding permutations, as we will see in \Cref{lem:i1isRmax} and \Cref{lem:nmrAscentPlus1IsNmrIdr}. Note that \Cref{lem:i1isRmax} has been briefly mentioned in \cite{do}, but here we will look at a formal proof of it.

\begin{lemma}\label{lem:i1isRmax}
\emph{\citet{do}}. Given $\pi\in\Av(231)$, then the first inverse descent run $I_1$ is $\Rmax(\pi)$.
\end{lemma}

\begin{proof}
Without knowing that $\pi\in\Av(231)$, we can still see that $\pi(I_1)=\{t, t-1, \dots, \pi_n\}$ ($n$ is the length of $\pi$). We first prove that $t=n$, then we prove $\Rmaxl(\pi)=\{n, n-1, \dots, \pi_n\}$.

Suppose that $t<n$, i.e. $t+1\notin\pi(I_1)$, then $t+1$ appears either before or after $t$. If it is before $t$, the inverse descent run can be extended and thus not maximal. If it is after $t$, then $t(t+1)\pi_n$ is an occurrence of 231. We reach a contradiction in both cases, hence $t+1\in\pi(I_1)$, which is another contradiction to the assumption $t+1\notin\pi(I_1)$. Therefore, $t=n$, or in other words, $\pi(I_1)=\{n,n-1,\dots,\pi_n\}$.

We already know that $\pi_n\in\Rmaxl(\pi)$. If $\pi_n=n$ then indeed $\Rmaxl(\pi)=\{n\}$. Otherwise, suppose that $\pi_n+x\in\Rmaxl(\pi)$ and $\pi_n+x+1\leq n$. When $x=0$, $\pi_n+x$ appears after $\pi_n+x+1$ in the permutation. The property also holds when $x>0$, otherwise $(\pi_n+x)(\pi_n+x+1)\pi_n$ is an occurrence of 231. Furthermore, the letters that appear between $\pi_n+x+1$ and $\pi_n+x$ in $\pi$ are all less than $\pi_n+x$, otherwise it would cause an occurrence of 231. Therefore, $\pi_n+x+1\in\Rmaxl(\pi)$. By induction, $\Rmaxl(\pi)=\{\pi_n, \pi_n+1, \dots, n\}$.
\end{proof}


\begin{lemma}\label{lem:nmrAscentPlus1IsNmrIdr}
Given $\pi\in\Av(231)$, $i$ is an ascent if and only if there exists $j > 1$ such that $i=\max(I_j)$.
\end{lemma}

\begin{proof}
Suppose that $i < n$ is an ascent, we can always find $j \leq 1$ such that $i \in I_j$. Suppose that $i\neq \max(I_j)$, then for some $i' > i$ and $i'\in I_j$, we have $\pi_i > \pi_{i'}$, thus $(i, i+1, i')$ is a 231-occurrence, contradict to the original assumption $\pi\in\Av(231)$.
\end{proof}

The following properties were also mentioned in \cite{do}, and the proof is quite straightforward, so we only present them here as a remark.

\begin{remark}
\emph{\citet{do}}. Given $\pi\in\Av(231)$:

\begin{itemize}
    \item if we remove $\pi_i$ where $i\in I_1$, we obtain a new permutation $\pi'\in\Av(231)$ having $k-1$ \idrs $I_1', I_2', \dots, I_{k-1}'$ and $\pi(I_2)=\pi'(I_1')$.
    \item if $u < v$ then $\pi(I_u)>\pi(I_v)$.
    \item $\{\max(I_1), \max(I_2),\dots, \max(I_k)\}=\Rmin(\pi)$.
\end{itemize}

\end{remark}

Given two set of integers $A$ and $B$. We have a couple of definitions:

\begin{definition}
$B$ is said to be nested in $A$ if there exists $x, y\in A$, $x < z$, such that $\forall y\in A$ we have $y < x$ or $y > z$ and $\forall y\in B$ we have $x < y < z$.
\end{definition}

\begin{definition}
$A$ and $B$ are said to be disjoint if $\max(A) < \min(B)$ or $\max(B) < \min(A)$.
\end{definition}

\begin{eg}
Given $A=\{1,2,8\}$,  $B=\{4,6,7\}$, $C=\{3,5\}$, $D=\{9\}$, then $B$ and $C$ are both nested in $A$; $A$ and $D$ are disjoint; $B$ and $C$ are neither nested in nor disjoint from each other.
\end{eg}

The following proposition from \cite{do} is also a fundamental property of \idr, which we will need to show if a permutation is 231-avoiding or not (see \Cref{lem:form2(13)}).

\begin{proposition}\label{nestedOrDisjoint}
\emph{\citet{do}}. Let $I_u$ and $I_v$ be two \idrs of $\pi$, then $\pi\in\Av(231)$ if and only if for any pair $u<v$ we have $I_u$ and $Iv$ are either nested or disjoint.
\end{proposition}

\subsection{Bijection $\theta': \Av'(231)\rightarrow\Av'(231)$ that transforms $\Asc$ to $\Atop$}\label{bijection theta'}

The bijection $\theta': \Av'(231)\rightarrow\Av'(231)$ presented in \citet{do} is not a direct transformation from $\Av'(231)$ to itself, but rather, it is a composition of two bijections $\theta'_1$ and $\theta'_2$, as shown below:

\begin{center}
    a permutation in $\Av'(231)$\\
    $\downmapsto$ $\theta'_1$\\
    a pair of \textit{consistent sequences}\\
    $\downmapsto$ $\theta'_2$\\
    a permutation in $\Av'(231)$.
\end{center}
where \textit{consistent sequences} is given by \Cref{def:consistent seq}. The bijection $\theta'_1$ is backed by \Cref{prop2}, which has been perfectly proven in \cite{do}. At the same time, $\theta'_2$ is backed by \Cref{prop3}; however, their constructive proof in \cite{do} was incorrect. Thus in this section we are going to revisit \Cref{prop3} and find a correct implementation of $\theta'_2$, thereby confirming the validity of $\theta'$ (\Cref{thm:ascatop}).

Let us take a look at the mentioned \Cref{def:consistent seq}, \Cref{prop2}, and \Cref{prop3}:

\begin{definition}\label{def:consistent seq}
\emph{\citet{do}}. Two sequences of positive integers $(c_1, c_2, \dots, c_k)$ and $(m_1, m_2, \dots, m_k)$ are said to be \emph{consistent} if these conditions hold:
\begin{itemize}
    \item $c_1 \leq 2 $
    \item $c_1+c_2+\dots+c_k=n$
    \item $n=m_1 > m_2 > \dots > m_k$
    \item $m_i > c_i + c_{i+1} + \dots + c_k$ (for $1<i\leq k$)
\end{itemize}
\end{definition}

\begin{proposition}\label{prop2}
\emph{\citet{do}}. Given a consistent pair of sequences $(c_1, c_2, \dots, c_k)$ and $(m_1, m_2, \dots, m_k)$, there is a unique permutation $\pi\in\Av'(231)$ such that:
\begin{itemize}
    \item $|I_i|=c_i$, where $I_i$ is the $i$-th \idr of $\pi$.
    \item $\Asc(\pi)=\{m_2,m_3,\dots,m_k\}$.
\end{itemize}
\end{proposition}

\begin{proposition}\label{prop3}
\emph{\citet{do}}. Given a consistent pair of sequences $(c_1, c_2, \dots, c_k)$ and $(m_1, m_2, \dots, m_k)$, there is a unique permutation $\pi\in\Av'(231)$ such that:
\begin{itemize}
    \item $|I_i|=c_i$, where $I_i$ is the $i$-th \idr of $\pi$.
    \item $\Atop(\pi)=\{m_2,m_3,\dots,m_k\}$.
\end{itemize}
\end{proposition}

In order to prove \Cref{prop3}, we will present a unique construction of $\pi$ from a pair of consistent sequences, such that $\pi$ abides the given conditions. \Cref{alg} is our approach where we use two stacks to construct a 231-permutation. Operations on stacks include top (reading the top without modifying the stack), pop (reading the top and remove it from the stack), and transfer (pop from stack $\calA$ and push that to stack $\calB$). There are three main steps in the algorithm that need to be focused, namely at Line \ref{st1} (abbreviated as L\ref{st1}), L\ref{st2}, and L\ref{st3}.

\begin{algorithm}[h]
\caption{Compute $\pi$ satisfying \Cref{prop3}}\label{alg}
\begin{algorithmic}[1]
\Procedure {Transfer}{$\calA$: Stack<Int>, $\calB$: Stack<Int>}
\State top $\leftarrow \calA$.pop()
\State $\calB$.push(top)
\EndProcedure
\Statex
\Procedure {ComputePermutation}{$c$: Array<Int>, $m$: Array<Int>}
\State $(n, p, k) \leftarrow (m_1, n, c.$size())
\For{$i\leftarrow 1,2,\dots, k$}
\For{$j\leftarrow 1,2,\dots, c_i$}\label{st1}\Comment Push $(n-\sum_{t=1}^{i-1} c_t),\dots, (n-\sum_{t=1}^{i}c_t +1)$ into $\calA$.
\State $\calA$.push($p$)
\State $p\leftarrow p-1$
\EndFor\label{endst1}
\While{$\calA$ is not empty and $\calA$.top() $\notin \{m_2, m_3,\dots, m_k\}$}\label{st2}
\State \textsc{Transfer}($\calA$, $\calB$)
\EndWhile\label{endst2}
\If {$\calA$ is not empty}\label{st3}
\State \textsc{Transfer}($\calA$, $\calB$)
\EndIf\label{endst3}
\EndFor
\EndProcedure
\end{algorithmic}
\end{algorithm}

        

We can briefly discuss the idea of \Cref{alg} as follows. In each for-loop on variable $i$, we push the entries of an \idr to $\calA$ (L\ref{st1}) so that when they are transferred to $\calB$ they will preserve that order and thereby forming an \idr in the resulting permutation. While transferring entries from $\calA$ to $\calB$, we will stop at a value if it is one of the required ascent tops (L\ref{st2} and L\ref{st3}), and go to the next for-loop to receive the entries of the next \idr. The entries pushed to $\calA$ are in descending order, so the next element transferred to $\calB$ will make the element right under it an ascent top in the resulting permutation. To further understand the algorithm, let us take an example.

\begin{eg}
Find $\theta'(7642135)$.

At the beginning of this \Cref{sec:majmakl}, we said that $\theta$ is a composition of $\theta'_1$ and $\theta'_2$. The first step, $\theta'_1$, is to find the consistent pair of sequences corresponding to the given permutation as discussed in \Cref{prop2}. In order to do so, we only need to identify the ascent set and the lengths of the \idrs.
\begin{itemize}
    \item Three \idrs, $\{1,2,7\}$, $\{3,6\}$, $\{4,5\}$. That would make $c=(3,2,2)$.
    \item Set of ascents as $\{5,6\}$. That would make $m=(7,6,5)$.
\end{itemize}

In the second step, $\theta'_2$, we convert the pair of sequences into a permutation by applying the procedure \textsc{ComputePermutation} in \Cref{alg}. We demonstrate the process as below:

\newcommand{\stack}[3]{
    \begin{tabular}{c*{#1}{|c}}
        \cline{1-#1}
         #2 & #3\\
        \cline{1-#1}
    \end{tabular}
}
\newcommand{\stackA}[2]{\stack{#1}{#2}{$\calA$} }
\newcommand{\stackB}[2]{\stack{#1}{#2}{$\calB$} }
\newcommand{\ab}[4]{
    \begin{tabular}[t]{l}
        \stackA{#1}{#2}\rule{0pt}{2.8ex} \\[0.1cm]
        \stackB{#3}{#4}\rule[-1.8ex]{0pt}{0pt} 
    \end{tabular}
}
\begin{center}
\begin{tabular}{p{1.4cm}|p{6cm}|l}
    \hline
    & \textbf{Action} & \textbf{Outcome}\\
    \hline
    Initially, & $\calA$ and $\calB$ is empty. Tops of the stacks are the open ends. & \ab{1}{}{1}{}\\\hline
    Iteration $i=1$ & L\ref{st1}: Because $c_i=3$, we sequentially push $7,6,5$ to $\calA$. & \ab{4}{& 5 & 6 & 7}{1}{}\\\cline{2-3}
    & L\ref{st2}: $\mathrm{top}(\calA)=5\in M$, do nothing. & (unchanged) \\\cline{2-3}
    & L\ref{st3}: Transfer the top (5) of $\calA$ to $\calB$. & \ab{3}{& 6 & 7}{2}{& 5} \\\hline
    Iteration $i=2$ & L\ref{st1}: Because $c_i=2$, we sequentially push $4,3$ to $\calA$. & \ab{5}{& 3 & 4 & 6 & 7}{2}{& 5}\\\cline{2-3}
    & L\ref{st2}: $\mathrm{top}(\calA)=3\notin M$, we transfer it to $\calB$. & \ab{4}{& 4 & 6 & 7}{3}{& 3 & 5} \\\cline{2-3}
    & L\ref{st2}: $\mathrm{top}(\calA)=4\notin M$, we transfer it to $\calB$. & \ab{3}{& 6 & 7}{4}{& 4 & 3 & 5} \\\cline{2-3}
    & L\ref{st2}: $\mathrm{top}(\calA)=6\in M$, do nothing. & (unchanged) \\\cline{2-3}
    & L\ref{st3}: Transfer the top (6) of $\calA$ to $\calB$. & \ab{2}{& 7}{5}{& 6 & 4 & 3 & 5} \\\hline
    Iteration $i=3$ & L\ref{st1}: Because $c_i=2$, we sequentially push $2,1$ to $\calA$. & \ab{4}{& 1 & 2 & 7}{5}{& 6 & 4 & 3 & 5}\\\cline{2-3}
    & L\ref{st2}: $\mathrm{top}(\calA)=1\notin M$, we transfer it to $\calB$. & \ab{3}{& 2 & 7}{6}{& 1 & 6 & 4 & 3 & 5} \\\cline{2-3}
    & L\ref{st2}: $\mathrm{top}(\calA)=2\notin M$, we transfer it to $\calB$. & \ab{2}{& 7}{7}{& 2 & 1 & 6 & 4 & 3 & 5} \\\cline{2-3}
    & L\ref{st2}: $\mathrm{top}(\calA)=7\notin M$, we transfer it to $\calB$. & \ab{1}{}{8}{& 7 & 2 & 1 & 6 & 4 & 3 & 5} \\\hline
\end{tabular}
\end{center}

At the final step, reading stack $\calB$ from top to bottom gives us $\theta'(7642135)=7216435$. We can confirm that $\Asc(7642135)=\Atop(7216435)$ and the \idrs' lengths are the same.
\end{eg}

In order to show the correctness of \Cref{alg}, we will show that the algorithm is well-defined (\Cref{hasLengthN} and \Cref{avoid231}) and the result $\pi$ satisfies the conditions in \Cref{prop3} (\Cref{trueAtopAndIdrs}).

\begin{lemma}\label{hasLengthN}
After \Cref{st2} of \Cref{alg}, $\calA$ is empty if and only if $i=k$. As a result, $\pi$ is an $n$-permutation.
\end{lemma}
\begin{proof}
First, consider the case $i<k$, we prove that $\calA$ is always non-empty after \Cref{st2}. Before \Cref{st1}:

\begin{itemize}
    \item If there is an entry $x\in\calA$ such that $x\in M$, then in \Cref{st2} we cannot empty the whole stack (due to the condition that top($\calA)\notin M$).
    \item If there is no such entry $x$, then we realize that $m_2, m_3, \dots, m_i$ are among the $\sum_{t=1}^{i-1}c_t$ largest elements of $[n]$. and were already pushed to $\calB$. It follows that $m_{i+1}$ is not among them, i.e. $m_{i+1}\leq n-\sum_{t=1}^{i-1}\Leftrightarrow m_{i+1}\leq \sum_{t=i}^k c_t$. Furthermore, the consistent pair gives us that $m_{i+1} > \sum_{t=i+1}^k c_t$, thus $n-\sum_{t=1}^i<m_{i+1}\leq n-\sum_{t=1}^{i-1}$. This means that $m_{i+1}$ will be pushed to $\calA$ in step $\Cref{st1}$, which will prevent \Cref{st2} from emptying $\calA$.
\end{itemize}

Therefore, when $i<k$, $\calA$ is always non-empty after \Cref{st2}. As a result, in each iteration we managed to transfer from $\calA$ to $\calB$ an entry $x\in M$, and in the last iteration ($i=k$), we simply have no such entries left. The condition top($\calA)\notin M$ will remain always true, and \Cref{st2} will empty the whole stack $\calA$.
\end{proof}

\begin{lemma}\label{avoid231}
The result $\pi$ of \Cref{alg} is 231-avoiding.
\end{lemma}

\begin{proof}
Suppose that it contains a 231-occurrence $(x,y,z)$. We know for a fact that $\calA$ is always monotonic increasing (from top to bottom). When $\pi_z$ was popped from $\calA$ and pushed to $\calB$, all entries larger than $\pi_z$, except those already in $\calB$, are all in $\calA$. If we ever push more elements to $\calA$, they are certainly smaller than $\pi_z$. Since $x<y$,  $\pi_y$ must have been popped from $\calA$ before $\pi_x$, so $\pi_y$ must be closer to the top than $\pi_x$. However, we have that $\pi_y>\pi_x$ and it contradicts the monotonic increasing property of $\calA$.
\end{proof}

\begin{lemma}\label{trueAtopAndIdrs}
The result $\pi$ of \Cref{alg} has:
\begin{itemize}
    \item $\Atop(\pi)=\{m_2, m_3,\dots, m_k\}$
    \item $k$ \idrs $I_1, I_2, \dots, I_k$ (descending ordered by the maximum element of each set)
    \item $I_j=c_j$ (with $1\leq j\leq k$)
\end{itemize}
\end{lemma}
\begin{proof}
Suppose that $m_j$ ($j>1$) is pushed to $\mathscr{B}$ in \Cref{st3} of iteration $i<k$. At iteration $i+1$, more entries are pushed to $\mathscr{A}$ and it's easy to see that all of these entries are smaller than $m_j$, and the next entry transferred to $\mathscr{B}$ will be one of them. Therefore $m_j$ is indeed an ascent-top. On the other hand, given $x\notin \{m_2, m_3,\dots, m_k\}$, it must have been pushed to $\mathscr{B}$ at \Cref{st2}. Due to the top-to-bottom increasing order of $\mathscr{A}$, the next letter pushed to $\mathscr{B}$ after $x$ will be larger than $x$, making $x$ a descent bottom, rather than an ascent top. Therefore, $\Atop=\{m_2,m_3,\dots,m_k\}.$

Because $\pi$ has $k-1$ ascent tops, it has $k$ \idrs (\cref{lem:nmrAscentPlus1IsNmrIdr}). Moreover, for some $i\leq k$, we pushed $(n-\sum_{t=1}^{i-1} c_t),(n-\sum_{t=1}^{i-1} c_t-1),\dots, (n-\sum_{t=1}^{i}c_t +1)$ sequentially in \Cref{st1} of iteration $i$. Therefore, they will appear in that order in $\pi$, forming an \idr of length $c_i$.
\end{proof}

The proof of uniqueness will be skipped, as it is routine to count that, given $n$, the number of consistent pairs of sequences (with respect $n$ and $k\leq n$) is equal to number of 231-avoiding permutations of length $n$ and starting with $n$.

As a result of \Cref{prop2} and \Cref{prop3}, the defined bijection $\theta'$ proves that \Cref{thm:ascatop} is a true statement. Now, we are interested in the original problem of \Cref{sec:majmakl}, the equidistribution of maj and makl on $\Av(231)$. As we mentioned, after finding $\theta'$, we will define $\theta$ in terms of $\theta'$; or more precisely, we define it as a \textit{direct sum} involving $\theta'$.

\begin{definition}
Given two permutations $\alpha$ and $\beta$, their \textit{direct sum} $\pi=\alpha\oplus\beta$ of length $|\alpha|+|\beta|$ is defined as:

$$
\pi_i=\left\{\begin{matrix}
    \alpha_i & \text{if} & i\leq |\alpha|\\
    \beta_{i-|\alpha|}+|\alpha| & \text{if} & i > |\alpha|
\end{matrix}\right.
$$
\end{definition}

Given $\pi\in\Av(231)$, let $\pi^{(1)}, \pi^{(2)}, \dots, \pi^{(t)}\in \Av'(231)$ such that
$\pi=\pi^{(1)}\oplus \pi^{(2)}\oplus \dots\oplus \pi^{(t)}$. Define $\theta$ as:
$\theta(\pi)=\theta'(\pi^{(1)})\oplus\theta'(\pi^{(2)})\oplus\dots\oplus\theta'(\pi^{(t)})$, then we have this result from \cite{do}:

\begin{theorem}\label{thm:majmakl}\emph{\citet{do}}.
$\theta$ transforms $\maj$ to $\makl$ on $\Av(231)$.
\end{theorem}

In the next section, we will see how \Cref{thm:majmakl} will help us prove the equidistribution of bast and foze on Av(312).

\section{Equidistribution of bast and foze on Av(312)}\label{sec:bastfoze}

The equidistribution of base and foze on Av(312) has been listed as a conjecture in \citet{amini}, and in this section we will prove that it is indeed correct. Previously, we have shown an equidistribution on Av(231) through bijection $\theta$; and notice how $c\circ r(231)=312$ and $c\circ r(312)=231$. This sparks an idea that we could convert a permutation in Av(312) to one in Av(231) through $c\circ r$, then apply $\theta$, then convert the result back to one in Av(312) through $c\circ r$. We have the following illustration and theorem:
$$\Av(312) \xrightarrow{c\ \circ\ r} \Av(231) \xrightarrow{\theta}\Av(231)\xrightarrow{c\ \circ\  r}\Av(312)$$

\begin{theorem}\label{thm:bastfoze}
The bijection composition $\Theta=c\circ r\circ \theta\circ c\circ r$ transforms $\bast$ to $\foze$ on $\Av(312)$.
\end{theorem}

In order to prove it, let us take a look at this lemma first:

\begin{lemma}\label{lem:form2(13)}
    Given $\pi\in\Av(231)$, the number of $2\underline{13}$-occurrences is $n-\rmax(\pi)-\rmin(\pi)+1$.
\end{lemma}
\begin{proof}
Let the \idrs be $I_1, I_2, \dots I_k$. Let $(x, y, y+1)$ be a possible occurrence of $2\underline{13}$. It is easy to notice that $y$ is an ascent and by \Cref{lem:nmrAscentPlus1IsNmrIdr} there is a value $1<j\leq k$ such that $y=\max(I_j)$. Let $i,t$ be integers where $y+1\in I_{i}$ and $x\in I_t$, then $i<j$.

If $t = j$, we immediately have that $\pi_y<\pi_x$ and $\pi_x<\pi_{y+1}$, thus $(x, y, y+1)$ is an occurrence of $2\underline{13}$. Consider other cases:

\begin{itemize}
    \item $t \leq i \Rightarrow \pi_x > \pi_{y+1} \Rightarrow (x,y,y+1)$ is not an occurrence of $2\underline{13}$.
    \item $i < t < j \Rightarrow \max(I_i) > \max(I_t) > \max(I_j)=y$. If $y+1=\max(I_i)$, then this case would never happen because no integers exist between $y$ and $y+1$. Suppose otherwise, consider four positions $x < y+1 < \max(I_t) < \max(I_i)$. They are respectively belong to $I_t, I_i, I_t, I_i$, making $I_i$ and $I_t$ neither nested nor disjoint, contradict to \Cref{nestedOrDisjoint}. Therefore this case would never happen.
    \item $t > j \Rightarrow \pi_x < \pi_y \Rightarrow (x,y,y+1)$ is not an occurrence of $2\underline{13}$.
\end{itemize}

$\Rightarrow$ For each ascent $y=\max(I_j)$ there are $|I_j| - 1$ possible values of $x$ such that $(x,y,y+1)$ is an occurrence of $2\underline{13}$.

Summing up that formula across all $1 < j\leq k$, we obtain the number of $2\underline{13}$ occurrences is:

\begin{align*}
2\underline{13}(\pi)&=\sum_{j=2}^k (|I_j|-1)=\sum_{j=2}^k |I_j|-(k-1)\\
&=n-|I_1|-k+1=n-\rmax(\pi)-\rmin(\pi)+1    
\end{align*}
\end{proof}

Now that the lemma is confirmed, let us go back to the original problem, \Cref{thm:bastfoze}.

\begin{proof}\emph{\Cref{thm:bastfoze}}.
From \Cref{thm:majmakl} we know that $\theta$ transforms maj to makl, i.e. given $\theta: \pi\mapsto\tau$ then $\maj(\pi)=\makl(\tau)$. It also follows that $\Rmax(\pi)=\Rmax(\tau)$ and $\Rmin(\pi)=\Rmin(\tau)$ (\cite{do}), and consequently $2\u{13}(\pi)=2\u{13}(\tau)$. 

Using the results in \Cref{tab:mahon}, we have that:
\begin{align*}
    \maj(\pi) &= \makl(\tau)\\
    \Leftrightarrow (1\u{32} + 2\u{31} + 3\u{21} + \u{21})(\pi) &= (1\u{32} + 2\u{31} + \u{32}1 + \u{21})(\tau)\\
    \Leftrightarrow (1\u{32} + 3\u{21} + \u{21})(\pi) &= (1\u{32} + \u{32}1 + \u{21})(\tau) & (\pi,\tau\in\Av(231))\\
    \Leftrightarrow (2\underline{13}+1\underline{32}+3\underline{21}+ \underline{21})(\pi)&=(1\underline{32}+\underline{32}1+2\underline{13}+\underline{21})(\tau) & (2\u{13}(\pi)=2\u{13}(\tau))\\
\end{align*}

Applying $c\circ r$ to $\pi$ and $\tau$ as well as all vincular patterns gives us a new equality:

\begin{align*}
    (\underline{13}2 + \underline{21}3 + \underline{32}1 + \underline{21})(c\circ r(\pi))&=(\underline{21}3 + 3\underline{21} + \underline{13}2 + \underline{21})(c\circ r(\tau))\\
    \bast(c\circ r(\pi))&=\foze(c\circ r(\tau))
\end{align*}

It is obvious that $\Theta\left(c\circ r(\pi)\right)=c\circ r(\tau)$. Therefore $\Theta$ transforms $\bast$ to $\foze$ on $\Av(312)$.
\end{proof}

The proof above is a combinatorial proof since it took advantage of the vincular pattern expressions and counted the occurrences of the patterns to draw conclusion about the equidistribution. We will see how the combinatorial approach works with another equidistribution in the section below.

\section{Equidistribution of foze$''$ on Av(312) and inv on Av(321)}\label{sec:foze''inv}

The Simion-Schmidt bijection $\psi: \Av(312)\rightarrow\Av(321)$ transforms a permutation $\pi$ into $\tau$ by keeping the left-to-right maxima in place and rearrange other letters in increasing order. In this section, we will prove the following theorem:

\begin{theorem}\label{thm:simionschmidt}
$\psi$ transforms $\foze''$ on $\Av(312)$ to $\inv$ on $\Av(321)$.
\end{theorem}

Despite the combinatorial proof below is longer than the induction proof in \cite{do}, it shows us a lot of beautiful formulas along the way (such as \Cref{invInAv321}, \ref{form23_1}, \ref{form21}, \ref{formfoze2}).


Recall that the pair $(x,y)$ is an inversion of $\tau$ if $x<y$ and $\tau_x > \tau_y$; and $\inv(\tau)$ is the number of inversions in $\tau$. Furthermore, let us denote $\inv(\tau, i)$ as the number of inversions $(x,y)$ in $\tau$ such that $x=i$. We have the first two lemmas:

\begin{lemma}\label{ltaFormula}
Given any permutation $\tau$ and $i\in\Lmax(\tau)$, then $\inv(\tau, i)=\tau_i - i$.
\end{lemma}
\begin{proof}
There are $\tau_i-1$ numbers smaller than $\tau_i$ in the permutation, and $i-1$ of them stand before $\tau_i$ because $i\in\Lmax(\tau)$. Therefore, there are $\tau_i-i$ numbers smaller than $\tau_i$ and stand after $\tau_i$.
\end{proof}

\begin{lemma}\label{invInAv321}
$(i,j)$ is an inversion in $\tau\in\Av(321)$ $\Rightarrow$ $i\in\Lmax(\tau)$.
\end{lemma}
\begin{proof}
Since $\tau\in\Av(321)$, removing the left-to-right maxima (which form an increasing subsequence) gives us an increasing sequence. In other words, it is \textit{merged} from two increasing sequences (see \citet{av321}). We quickly realize that in order for $(i,j)$ (where $i<j$) to be an inversion, $i$ and $j$ must belong to different subsequences, because otherwise $\tau_i<\tau_j$ would happen. In other words, one of them must be a left-to-right maximum. On the other hand, $j$ cannot be a left-to-right maxiumum because that would cause $\tau_i<\tau_j$. Therefore, $i$ must be a left-to-right maximum.
\end{proof}

Combining the result of \Cref{ltaFormula} and \Cref{invInAv321}, we can just count all inversions of $\tau\in\Av(321)$ by totalling $\inv(\tau,i)$ across all $i\in\Lmax(\tau)$.

\begin{lemma}\label{forminv}
Given $\tau\in\Av(321)$: $$\inv(\tau)=\sum_{i\in\Lmax(\tau)}(\tau_i-i)$$
\end{lemma}

\newcommand{\plmax}{\mathrm{plmax}}
\Cref{ltaFormula} gives us the formula of $\inv(\pi, i)$ for $i\in\Lmax(\pi)$ and any $\pi$. When $\pi\in\Av(312)$, we can extend the formula for any $i\in [n]$. Denote $\plmax(\pi, i)$ as the nearest left-to-right maximum that is not greater than $i$ (the letter \textit{p} stands for \textit{previous}).
\begin{lemma}
Given $\pi\in\Av(312)$ and some $i\in [n]$, then $$\inv(\pi, i) = \inv\left(\pi, \plmax(\pi, i)\right)-\left(i-\plmax(\pi, i)\right)$$
\end{lemma}
\begin{proof}
We can see that $\pi_i > \pi_{\plmax(\pi, i)}$ and $\pi\in\Av(312)$, thus for any $j>i$ we have $\pi_i < \pi_j$ $\Leftrightarrow$ $\pi_{\plmax(\pi, i)} < \pi_j$. Therefore, we only need to exclude $(i-\plmax(\pi, i))$ entries standing between position $\plmax(\pi, i)$ and $i$ (and $i$ itself, too).
\end{proof}

We know from \Cref{tab:mahon} that $\foze''(\pi)=(\underline{23}1 + \underline{31}2 + \underline{31}2 + \underline{21})(\pi)$. However, since $\pi\in\Av(312)$, $\foze''(\pi)=(\u{23}1+\u{21})(\pi)$. The two lemmas below will give the formula for each vincular pattern:

\begin{lemma}\label{form23_1}
Given $\pi\in\Av(312)$, then
$$\u{23}1(\pi)=\sum_{i\in\Lmax(\pi)}(\pi_i-i)-n+\lmax(\pi)$$
\end{lemma}
\begin{proof}
If $(x-1,x,y)$ is a $\u{23}1$-occurrence, then $x\in\Lmax(\pi)$; because if $x\notin\Lmax(\pi)$, then $(\plmax(\pi, x), x-1, x)$ would be a 312-occurrence.

For some $i\in\Lmax(\pi)\textbackslash\{1\}$, the number of $\u{23}1$-occurrences $(i-1, i, j)$ is $\inv(\pi, i-1)$, which, by \Cref{form23_1}, is:
\begin{align*}
    \inv(\pi, i-1)&=\inv(\pi,\plmax(\pi,i-1))-(i-1-\plmax(\pi, i-1)\\
    &=\pi_{\plmax(\pi,i-1)}-\plmax(\pi,i-1)-(i-1-\plmax(\pi,i-1))\\
    &=\pi_{\plmax(\pi,i-1)}-i+1\\
\end{align*}

Totalling that value across all $i\in\Lmax(\pi)\textbackslash\{1\}$, we have:

\begin{align*}
    \u{23}1(\pi)&=\sum_{i\in\Lmax(\pi)\textbackslash\{1\}}(\pi_{\plmax(\pi,i-1)}-i+1)\\
    &=\sum_{i\in\Lmax(\pi)\textbackslash\{1\}}\pi_{\plmax(\pi,i-1)} - \sum_{i\in\Lmax(\pi)\textbackslash\{1\}}i + \lmax(\pi)-1\\
    &=\sum_{i\in\Lmax(\pi)\textbackslash\{1\}}\pi_{\plmax(\pi,i-1)} - \sum_{i\in\Lmax(\pi)}i + \lmax(\pi)\\
    &=\sum_{i\in\Lmax(\pi)\textbackslash\{n\}}\pi_i - \sum_{i\in\Lmax(\pi)}i + \lmax(\pi)\\
    &=\sum_{i\in\Lmax(\pi)}\pi_i -n - \sum_{i\in\Lmax(\pi)}i + \lmax(\pi)
\end{align*}
\end{proof}

\begin{lemma}\label{form21}
Given $\pi\in\Av(312)$ and $1\leq y < n$, then
$$y\in\Des(\pi)\Leftrightarrow y+1\notin\Lmax(\pi)$$
which implies:
$$\u{21}(\pi)=\des(\pi)=n-\lmax(\pi)$$
\end{lemma}
\begin{proof}
It is natural that $y\in\Des(\pi)\Rightarrow y+1\notin\Lmax(\pi)$. On the other hand, when $y+1\notin\Lmax(\pi)$, suppose that $y$ is an ascent, then $(\plmax(\pi,y), y, y+1)$ is a 312-occurrence (contradiction). Therefore $y$ must be a descent.

Because $1\leq y<n$, there are $n-1$ possible values, among which $\lmax(\pi)-1$ values satisfying the condition $y+1\in\Lmax(\pi)$. Therefore the number of descents is $n-1-(\lmax(\pi)-1)=n-\lmax(\pi)$.
\end{proof}

From \Cref{form23_1} and \Cref{form21}, we can conclude that:

\begin{lemma}\label{formfoze2}
Given $\pi\in\Av(312)$:
$$\foze''(\pi)=\sum_{i\in\Lmax(\pi)}(\pi_i-i)$$
\end{lemma}

Recall that the Simion-Schmidt bijection $\psi: \pi\mapsto\tau$ preserves the set of left-to-right maxima, i.e. $\Lmax(\pi)=\Lmax(\tau)$. Knowing that $\pi\in\Av(312)$ and $\tau\in\Av(321)$, and the results of \Cref{formfoze2} and \Cref{forminv}, we can confirm that $\foze''(\pi)=\inv(\tau)$, thus \Cref{thm:simionschmidt} is true.

In the next chapter, we will conclude and summarize the contributions we made, together with some notes on the methods we took while working on the thesis.

\chapter{Conclusion}\label{chapter:conclusion}
\newcommand{\st}{\mathrm{st}}
Before concluding the thesis, we would like to put some emphasis on how we came up with the bijection for an equidistribution problem. Suppose we are given two statistics st$_1$ and st$_2$, and two permutation sets $S_n(\sigma_1)$ and $S_n(\sigma_2)$, which are conjectured to be equidistributed. For small $n$, we can group/partition $S_n(\sigma_1)$ (resp. $S_n(\sigma_2)$) into subsets and label them by the value that $\st_1$ (resp. $\st_2$) induce on them. Then, we would want to find a third statistic $\st_3$ that distributes evenly on any two subsets $A_1\subseteq S_n(\sigma_1)$ and $A_2\subseteq S_n(\sigma_1)$ that have the same partition-label. After that, we would look for a bijection that not only transforms $\st_1$ into $\st_2$, but also preserves $\st_3$. The more statistics $\st_3$ we found, the higher our chance to figure out what the bijection might be. We wrote a computer program to try out different candidate for $\st_3$, among a pool of about 20 well-known statistics (see more Appendix A of \citet{kitaev}).

For example, during working on the pair bast and foze on Av(312), we found that the unknown bijection could possibly preserve the statistics head, head\_i, last, lmax, and lmin. We then notice a bijection in \citet{do} that preserves lmax and lmin too, so we programmed the bijection and later found out that it also solves our equidistribution. Once we were confident with the bijection (by testing it with large permutations), we would start working on the proof. In another attempt to prove the equidistribution between makl and bast on Av(132), we also applied this method and found that the unknown bijection solution $\gamma$ could preserve one or more of these statistics: asc, des, head, lds, lir\_i, lmax, lmin, peak\_i, valley\_i, zeil (see \citet{kitaev} for the definitions). However, we have not found the bijection and the problem remained unsolved.

To sum up, we have proven one new equidistribution and refined the solutions of two established equidistributions. We have shown the alternative construction of the bijection $\theta'$ proposed by \citet{do}, which transforms the statistic Asc to Atop in Av'(231) (\Cref{sec:majmakl}). From there, we took some further step to prove the conjectured equidistribution of bast and foze on on Av(312), with the help of \Cref{lem:form2(13)} (\Cref{sec:bastfoze}). We also gave a combinatorial proof for the equidistribution problem between $\foze''$ on Av(312) and inv on Av(321) (\Cref{sec:foze''inv}) through various formulas. Below is a summary of proven equalities, which we believe will help other researchers solving more equidistribution problems.

\begin{center}
    \begin{tabular}{c|rl|c}
        \hline
        \textbf{References} & \multicolumn{2}{c|}{\textbf{Equality}} & \textbf{Precondition} \\\hline
        \Cref{lem:form2(13)} & $2\u{13}(\pi)$ & $=n-\rmax(\pi)-\rmin(\pi)+1$ & $\pi\in\Av(231)$\\
        \Cref{invInAv321} & $\inv(\tau)$ & $=\sum_{i\in\Lmax(\tau)}(\tau_i-i)$ & $\tau\in\Av(321)$\\
        \Cref{form23_1} & $\u{23}1(\pi)$ & $=\sum_{i\in\Lmax(\pi)}(\pi_i-i)-n+\lmax(\pi)$ & $\pi\in\Av(312)$\\
        \Cref{form21} & $\u{21}(\pi)$ & $=\des(\pi)=n-\lmax(\pi)$ & $\pi\in\Av(312)$\\
        \Cref{formfoze2} & $\foze''(\pi)$ & $=\sum_{i\in\Lmax(\pi)}(\pi_i-i)$ & $\pi\in\Av(312)$\\
        \hline
    \end{tabular}
\end{center}

\appendix{}

\chapter{Popular statistics}
\begin{table}[h]
    \centering
    \begin{tabular}{c|l|l}
        \hline
        \textbf{Notation} & \textbf{Description} & \textbf{Definition}\\\hline
        Asc & the set of ascents & $\{i<n \vert \pi_i<\pi_{i+1}\}$\\
        asc & number of ascents & $|\Asc(\pi)|$\\
        Abot & the set of ascent bottoms & $\{\pi_i \vert i<n \text{ and } \pi_i<\pi_{i+1}\}$\\
        Atop & the set of ascent tops & $\{\pi_{i+1} \vert i<n \text{ and } \pi_i<\pi_{i+1}\}$\\\hline
        Des & the set of descents & $\{i<n \vert \pi_i>\pi_{i+1}\}$\\
        des & number of descents & $|\Des(\pi)|$\\
        Dbot & the set of descent bottoms & $\{\pi_{i+1} \vert i<n \text{ and } \pi_i>\pi_{i+1}\}$\\
        Dtop & the set of descent tops & $\{\pi_i \vert i<n \text{ and } \pi_i>\pi_{i+1}\}$\\\hline
        Lmax & the set of left-to-right maxima & $\{\pi_i\vert \pi_i>\pi_j \text{ for any } j<i\}$\\
        Lmaxl & the entries at the left-to-right maxima & $\pi(\Lmax(\pi))$\\
        lmax & number of left-to-right maxima & $|\Lmax(\pi)|$\\\hline
        Lmin & the set of left-to-right minima & $\{\pi_i\vert \pi_i<\pi_j \text{ for any } j<i\}$\\
        Lminl & the entries at the left-to-right minima & $\pi(\Lmin(\pi))$\\
        lmin & number of left-to-right minima & $|\Lmin(\pi)|$\\\hline
        Rmax & the set of right-to-left maxima & $\{\pi_i\vert \pi_i>\pi_j \text{ for any } j>i\}$\\
        Rmaxl & the entries at the right-to-left maxima & $\pi(\Rmax(\pi))$\\
        rmax & number of right-to-left maxima & $|\Rmax(\pi)|$\\\hline
        Rmin & the set of right-to-left minima & $\{\pi_i\vert \pi_i<\pi_j \text{ for any } j>i\}$\\
        Rminl & the entries at the right-to-left minima & $\pi(\Rmin(\pi))$\\
        rmin & number of right-to-left minima & $|\Rmin(\pi)|$\\\hline
    \end{tabular}
    \caption{Some popular combinatorial and set-induced statistics on permutations}
    \label{tab:statistics}
\end{table}

\microtypesetup{protrusion=false}
\microtypesetup{protrusion=true}
\printglossaries
\bibliographystyle{plainnat}
\bibliography{thesis}

\end{document}